\documentclass[12pt,a4paper]{amsart}
\usepackage[text={16cm,24.3cm}]{geometry}
\usepackage{amsfonts, amsmath, amssymb, amsthm}
\usepackage{amscd}
\usepackage{url}
\usepackage{graphicx,color}
\usepackage{booktabs}
\usepackage{url}
\usepackage{xspace}
\usepackage[mathscr]{eucal}

\newtheorem{thm}{Theorem}[section]
\newtheorem{cor}[thm]{Corollary}
\newtheorem{lem}[thm]{Lemma}
\newtheorem{prop}[thm]{Proposition}

\theoremstyle{definition}
\newtheorem{example}[thm]{Example}
\newtheorem{remark}[thm]{Remark}

\newcommand{\K}{{\mathbb{K}}}
\newcommand{\R}{{\mathbb{R}}}

\newcommand{\Z}{{\mathbb{Z}}}

\newcommand{\Q}{{\mathbb{Q}}}
\newcommand{\PG}{{\mathrm{PG}}}
\newcommand{\GL}{{\mathrm{GL}}}
\newcommand{\PP}{{\mathbb{P}}}
\newcommand{\Gr}{{\mathrm{Gr}}}
\newcommand{\Dr}{{\mathrm{Dr}}}
\newcommand{\TA}{{\mathbb{T}}}
\newcommand{\TP}{{\mathbb{TP}}}

\providecommand\cF{{\mathscr{F}}}
\providecommand\cM{{\mathscr{M}}}
\providecommand\cX{{\mathscr{X}}}
\providecommand\cddlib{\texttt{cddlib}\xspace}
\providecommand\SoPlex{\texttt{SoPlex}\xspace}
\providecommand\Macaulay{\texttt{Macaulay2}\xspace}
\providecommand\Gfan{\texttt{Gfan}\xspace}
\providecommand\polymake{\texttt{polymake}\xspace}
\providecommand\homology{\texttt{homology}\xspace}
\providecommand{\Sym}{\operatorname{S}}
\providecommand{\MatroidPolytope}[1]{P_{#1}}
\providecommand{\SetOf}[2]{\left\{#1\vphantom{#2}\,\right.\left|\,\vphantom{#1}#2\right\}}
\providecommand{\smallSetOf}[2]{\{#1\,|\,#2\}}
\providecommand{\GF}{{\mathbb{GF}}}

\DeclareMathOperator{\conv}{conv}
\DeclareMathOperator{\rank}{rank}
\DeclareMathOperator{\initial}{in}
\providecommand{\tropical}[1]{{\mathbb{T}}(#1)}

\providecommand\cprime{$'$}
\providecommand{\bysame}{\leavevmode\hbox to3em{\hrulefill}\thinspace}

\graphicspath{ {pictures/} }
\title{How to Draw Tropical Planes}
\author[Herrmann, Jensen, Joswig, and Sturmfels]{Sven Herrmann \and 
Anders Jensen \and \\ Michael Joswig \and Bernd Sturmfels}

\address{Sven Herrmann, FB Mathematik, TU Darmstadt, 64289 Darmstadt,
  Germany} \email{sherrmann@mathematik.tu-darmstadt.de}

\address{Anders Jensen, Courant Research Center, Mathematisches
  Institut, Georg-August-Universit\"at, 37073 G\"ottingen, Germany}
\email{jensen@uni-math.gwdg.de}

\address{Michael Joswig, FB Mathematik, TU Darmstadt, 64289 Darmstadt,
  Germany} \email{joswig@mathematik.tu-darmstadt.de}

\address{Bernd Sturmfels, Department of Mathematics, UC Berkeley,
  Berkeley CA 94720, USA} \email{bernd@math.berkeley.edu}

\thanks{Sven Herrmann was supported by a Graduate Grant of TU
  Darmstadt.  Anders Jensen was supported by a Sofia Kovalevskaja
  prize awarded to Olga Holtz at TU Berlin.  Michael Joswig was
  supported by the DFG Research Unit ``Polyhedral Surfaces''.  Bernd
  Sturmfels was supported by an Alexander-von-Humboldt senior award at
  TU Berlin and the US National Science Foundation.}

\dedicatory{Dedicated to Anders Bj\"orner on the occasion of his 60th birthday.}

\begin{document}

\bibliographystyle{plain}

\maketitle

\begin{abstract}
  The tropical Grassmannian parameterizes tropicalizations of ordinary
  linear spaces, while the Dressian parameterizes all tropical linear
  spaces in $\TP^{n-1}$.  We study these parameter spaces and we
  compute them explicitly for $n \leq 7$.  Planes are identified with
  matroid subdivisions and with arrangements of trees.  These
  representations are then used to draw pictures.
\end{abstract}

\section{Introduction}

A line in tropical projective space $\TP^{n-1}$ is an embedded metric tree
which is balanced and has $n$ unbounded edges pointing into the coordinate
directions. The parameter space of these objects is the tropical Grassmannian
$\Gr(2,n)$. This is a simplicial fan \cite{SS}, known to evolutionary
biologists as the \emph{space of phylogenetic trees} with $n$ labeled leaves
\cite[\S 3.5]{ASCB}, and known to algebraic geometers as the \emph{moduli
  space of rational tropical curves} \cite{Mik}.

Speyer \cite{Spe1, Spe2} introduced higher-dimensional tropical linear
spaces. They are contractible polyhedral complexes all of whose maximal cells
have the same dimension $d-1$.  Among these are the realizable tropical linear
spaces which arise from $(d-1)$-planes in classical projective space
$\PP^{n-1}_\K$ over a field $\K$ with a non-archimedean valuation.  Realizable
linear spaces are parameterized by the tropical Grassmannian $\Gr(d,n)$, as
shown in \cite{SS}.  Note that, as a consequence of \cite[Theorem 3.4]{SS} and
\cite[Proposition 2.2]{Spe1}, all tropical lines $(d=2)$ are realizable.
Tropical Grassmannians represent compact moduli spaces of hyperplane
arrangements.  Introduced by Alexeev, Hacking, Keel, and Tevelev \cite{Ale,
  HKT, KT}, these objects are natural generalizations of the moduli space
$\overline{M}_{0,n}$.

In this paper we focus on the case $d=3$.  By a \emph{tropical plane} we mean
a two-dimensional tropical linear subspace of $\TP^{n-1}$. It was shown in
\cite[\S 5]{SS} that all tropical planes are realizable when $n \leq 6$. This
result rests on the classification of planes in $\TP^5$ which is shown in
Figure \ref{fig:planes}.  We here derive the analogous complete picture of
what is possible for $n=7$.  In Theorem \ref{thm:dim_bound}, we show that for
larger $n$ most tropical planes are not realizable.  More precisely, the
dimension of $\Dr(3,n)$ grows quadratically with $n$, while the dimension of
$\Gr(3,n)$ is only linear in $n$.

Tropical linear spaces are represented by vectors of Pl\"ucker
coordinates. The axioms characterizing such vectors were discovered two
decades ago by Andreas Dress who called them \emph{valuated matroids}.  We
therefore propose the name \emph{Dressian} for the tropical prevariety
$\Dr(d,n)$ which parameterizes $(d-1)$-dimensional tropical linear spaces in
$\TP^{n-1}$.  The purpose of this paper is to gather results about $\Dr(3,n)$
which may be used in the future to derive general structural information about
all Dressians and Grassmannians.
\begin{figure}[hp]
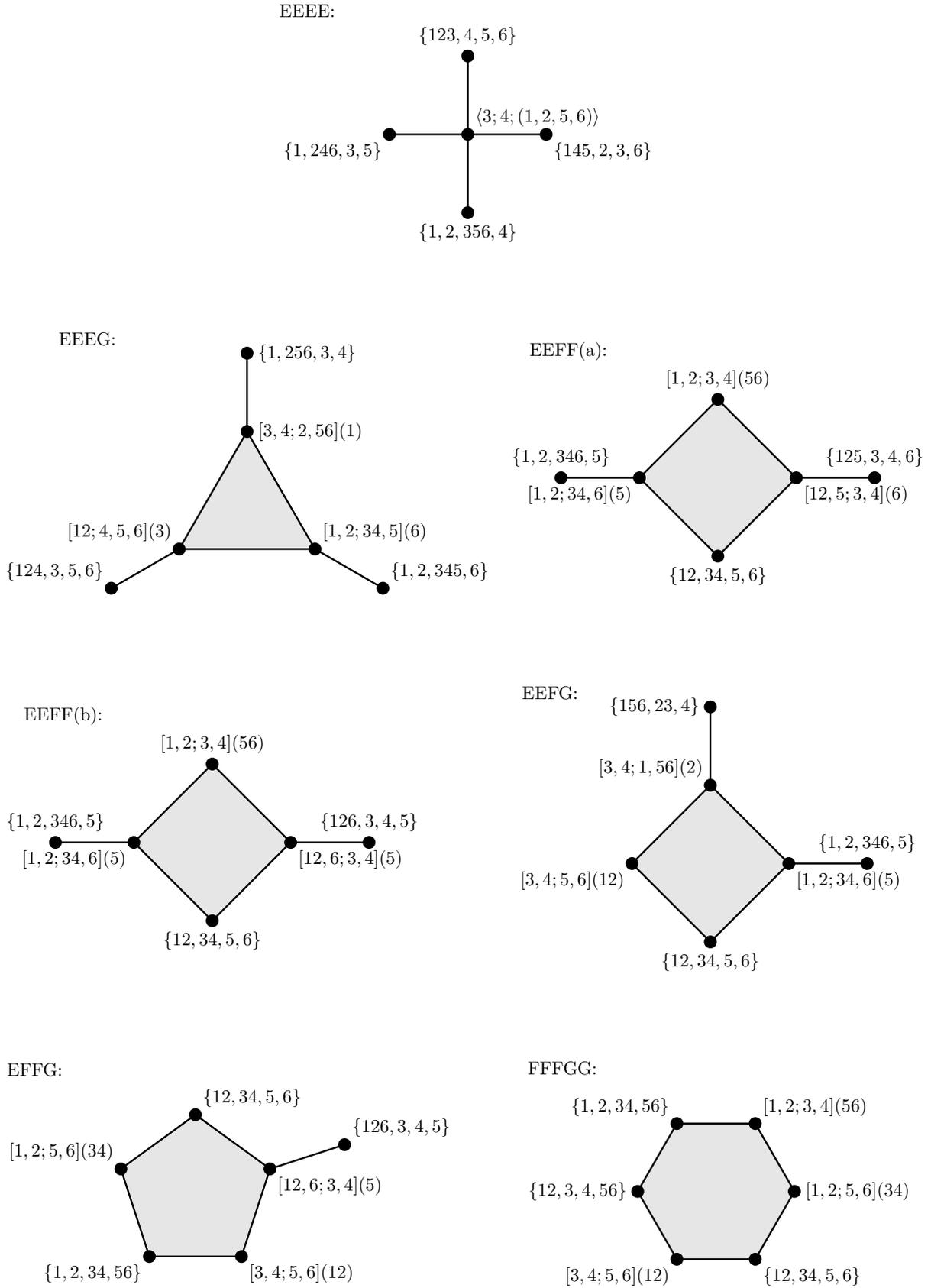
\centering

  \includegraphics[scale=0.91]{Planes.7}

  \vspace{1.5cm}

  \begin{minipage}[c]{.45\textwidth}\centering
    \includegraphics[scale=0.91]{Planes.6}
  \end{minipage}
  \hfill
  \begin{minipage}[c]{.45\textwidth}\centering
    \includegraphics[scale=0.91]{Planes.3}
  \end{minipage}

  \vspace{1.5cm}
  
  \begin{minipage}[c]{.45\textwidth}\centering
    \includegraphics[scale=0.91]{Planes.4}
  \end{minipage}
  \hfill
  \begin{minipage}[c]{.45\textwidth}\centering
    \includegraphics[scale=0.91]{Planes.5}
  \end{minipage}

  \vspace{1.5cm}

  \begin{minipage}[c]{.45\textwidth}\centering
    \includegraphics[scale=0.91]{Planes.2}
  \end{minipage}
  \hfill
  \begin{minipage}[c]{.45\textwidth}\centering
    \includegraphics[scale=0.91]{Planes.1}
  \end{minipage}
  
  \caption{The seven types of generic tropical planes in $\TP^5$.}
  \label{fig:planes}
\end{figure}

The paper is organized as follows. In Section~\ref{sec:computation} we review
the formal definition of the Dressian and the Grassmannian, and we present our
results on $\Gr(3,7)$ and $\Dr(3,7)$.  These also demonstrate the remarkable
scope of current software for tropical geometry.  In particular, we use
\Gfan~\cite{Gfan} for computing tropical varieties and
\polymake~\cite{polymake} for computations in polyhedral geometry.

Tropical planes are dual to regular matroid subdivisions of the hypersimplex
$\Delta(3,n)$. The theory of these subdivisions is developed in Section 3,
after a review of matroid basics, and this allows us to prove various
combinatorial results about the Dressian $\Dr(3,n)$.  With a specific
construction of matroid subdivisions of the hypersimplices which arise from
the set of lines in finite projective spaces over $\GF(2)$ these combinatorial
results yield the lower bound on the dimensions of the Dressians in
Theorem~\ref{thm:dim_bound}.

A main contribution is the bijection between tropical planes and arrangements
of metric trees in Theorem \ref{thm:arrangement}.  This bijection tropicalizes
the following classical picture.  Every plane $\PP_\K^{n-1}$ corresponds to an
arrangement of $n$ lines in $\PP_\K^2$, and hence to a rank-$3$-matroid on $n$
elements.  Lines are now replaced by trees, and arrangements of trees are used
to encode matroid subdivisions.  These can be non-regular, as shown in
Section~\ref{sec:trees}.  A key step in the proof of Theorem
\ref{thm:arrangement} is Proposition~\ref{prop:bar_sigma} which compares the
two natural fan structures on $\Dr(3,n)$, one arising from the structure as a
tropical prevariety, the other from the secondary fan of the hypersimplex
$\Delta(3,n)$.  It turns out that they coincide.  The
Section~\ref{sec:pictures} answers the question in the title of this paper,
and, in particular, it explains the seven diagrams in Figure \ref{fig:planes}
and their $94$ analogs for $n=7$.  In Section~\ref{sec:pappus} we extend the
notion of Grassmannians and Dressians from $\Delta(d,n)$ to arbitrary matroid
polytopes.

We are indebted to Francisco Santos, David Speyer, Walter Wenzel, Lauren
Williams, and an anonymous referee for their helpful comments.

\section{Computations}\label{sec:computation}

Let $I$ be a homogeneous ideal in the polynomial ring
$\K[x_1,\dots,x_t]$ over a field $\K$. Each vector $\lambda\in\R^t$
gives rise to a partial term order and thus defines an initial ideal
$\initial_\lambda(I)$, by choosing terms of lowest weight for each
polynomial in $I$.  The set of all initial ideals of $I$ induces a fan
structure on $\R^t$. This is the \emph{Gr\"obner fan} of $I$, which
can be computed using \Gfan~\cite{Gfan}.  The subfan induced by those
initial ideals which do not contain any monomial is the \emph{tropical
  variety} $\tropical{I}$.  If $I$ is a principal ideal then
$\tropical{I}$ is a \emph{tropical hypersurface}.  A \emph{tropical
  prevariety} is the intersection of finitely many tropical
hypersurfaces.  Each tropical variety is a tropical prevariety, but
the converse does not hold \cite[Lemma~3.7]{RGST}.

Consider a fixed $d\times n$-matrix of indeterminates. Then each
$d\times d$-minor is defined by selecting $d$ columns
$\{i_1,i_2,\dots,i_d\}$.  Denoting the corresponding minor
$p_{i_1\ldots i_d}$, the algebraic relations among all $ d\times
d$-minors define the \emph{Pl\"ucker ideal} $I_{d,n}$ in $\K[p_S]$,
where $S$ ranges over $\tbinom{[n]}{d}$, the set of all $d$-element
subsets of $[n]:=\{1,2,\dots,n\}$.  The ideal $I_{d,n}$ is a
homogeneous prime ideal.  The \emph{tropical Grassmannian} $\Gr(d,n)$
is the tropical variety of the Pl\"ucker ideal $I_{d,n}$.  Among the
generators of $I_{d,n}$ are the \emph{three term Pl\"ucker relations}
\begin{equation}\label{eq:3term}
  p_{Sij}p_{Skl} - p_{Sik}p_{Sjl} + p_{Sil}p_{Sjk} \, ,
\end{equation}
where $S\in\tbinom{[n]}{d-2}$ and $i,j,k,l\in [n]\backslash S$
pairwise distinct.  Here $Sij$ is shorthand notation for the set
$S\cup\{i,j\}$.  The relations (\ref{eq:3term}) do not generate the
Pl\"ucker ideal $I_{d,n}$ for $d \geq 3$, but they always suffice to
generate the image of $I_{d,n}$ in the Laurent polynomial ring
$\K[p_S^{\pm 1}]$.

The \emph{Dressian} $\Dr(d,n)$ is the tropical prevariety defined by
all three term Pl\"ucker relations.  The elements of $\Dr(d,n)$ are
the \emph{finite tropical Pl\"ucker vectors} of Speyer~\cite{Spe1}.  A
\emph{general tropical Pl\"ucker vector} is allowed to have $\infty$
as a coordinate, while a \emph{finite} one is not.  The three term
relations define a natural \emph{Pl\"ucker fan structure} on the
Dressian~$\Dr(d,n)$: two weight vectors $\lambda$ and $\lambda'$ are
in the same cone if they specify the same initial form for each
trinomial (\ref{eq:3term}).  In Sections 3 and 4 we shall derive an
alternative description of the Dressian $\Dr(d,n)$ and its Pl\"ucker
fan structure in terms of matroid subdivisions.

The Grassmannian and the Dressian were defined as fans in
$\R^{\binom{n}{d}}$.  One could also view them as subcomplexes in the
\emph{tropical projective space} $\TP^{\binom{n}{d}-1}$, which is the
compact space obtained by taking $\,(\R \cup
\{\infty\})^{\binom{n}{d}} \backslash \{(\infty, \ldots, \infty)\} \,$
modulo tropical scalar multiplication.  We adopt that interpretation
in Section \ref{sec:pappus}.  Until then, we stick to
$\R^{\binom{n}{d}}$.  Any polyhedral fan gives rise to an
\emph{underlying} (spherical) polytopal complex obtained by
intersecting with the corresponding unit sphere.  Moreover, the
Grassmannian $\Gr(d,n)$ and the Dressian $\Dr(d,n)$ have the same
$n$-dimensional lineality space which we can factor out.  This gives
pointed fans in $\R^{\binom{n}{d} - n}$.  For the underlying spherical
polytopal complexes of these pointed fans we again use the notation
$\Gr(d,n)$ and $\Dr(d,n)$.  The former has dimension $d(n-d)-n$, while
the latter is a generally higher-dimensional polyhedral complex whose
support contains the support of $\Gr(d,n)$.  For instance, $\Gr(2,5) =
\Dr(2,5)$ is the Petersen graph.  In the sequel we will discuss
topological features of $\Gr(d,n)$ and $\Dr(d,n)$.  In these cases we
always refer to the underlying polytopal complexes of these two fans
modulo their lineality spaces.  Each of the two fans is a cone over
the underlying polytopal complex (joined with the lineality space).
Hence the fans are topologically trivial, while the underlying
polytopal complexes are not.

It is clear from the definitions that the Dressian contains the
Grassmannian (over any field $\K$) as a subset of $\R^{\binom{n}{d}}$;
but it is far from obvious how the fan structures are related.  
Results of \cite{SS} imply that $\Gr(2,n) = \Dr(2,n)$ as fans and that
$\Gr(3,6) = \Dr(3,6)$ as sets.  Using computations with the software
systems \Gfan \cite{Gfan}, \homology \cite{homology}, \Macaulay
\cite{M2}, and \polymake \cite{polymake} we obtained the following
results about the next case $(d,n)=(3,7)$.

\begin{thm}
  \label{thm:Gr37}
  Fix any field $\K$ of characteristic different from 2.  The
  tropical Grassmannian $\Gr(3,7)$, with its induced Gr\"obner fan
  structure, is a simplicial fan with $f$-vector
  \[
  (721, 16800, 124180, 386155, 522585, 252000) \, .
  \]
  The homology of the underlying five-dimensional simplicial complex
  is free Abelian, and it is concentrated in top dimension:
  \[
  H_*\bigl( \Gr(3,7); \Z \bigr) \ = \ H_5\bigl( \Gr(3,7); \Z \bigr)
  \ = \ \Z^{7470} \, .
  \]
\end{thm}

The result on the homology is consistent with Hacking's theorem in
\cite[Theorem~2.5]{Hac}.Indeed, Hacking showed that if the tropical
compactification is sch\"on then the homology of the tropical variety is
concentrated in top dimension, and it is conjectured in \cite[\S1.4]{KT} that
the property of being sch\"on holds for the Grassmannian when $d=3$ and $n=7$;
see also \cite[Example 4.2]{Hac}.  Inspired by Markwig and Yu \cite{MY}, we
conjecture that the simplicial complex $\Gr(3,7)$ is shellable.


\begin{thm}
  \label{thm:Dr37}
  The Dressian $\Dr(3,7)$, with its Pl\"ucker fan structure, is a
  non-simplicial fan. The underlying polyhedral complex is
  six-dimensional and has the $f$-vector
  \[
  (616, 13860, 101185, 315070, 431025, 211365, 30) \, .
  \]
  Its $5$-skeleton is triangulated by the Grassmannian $\Gr(3,7)$, and
  the homology is
  \[
  H_*\bigl( \Dr(3,7); \Z \bigr) \ = \ H_5\bigl( \Dr(3,7); \Z \bigr)
  \ = \ \Z^{7440} \, .
  \]
\end{thm}

We note that the combinatorial and algebraic notions in this paper are
compatible with the geometric theory developed in Mikhalkin's book
\cite{Mik}. We here use ``$\min$'' for tropical addition, the set
$\,\TA^{k-1}=\R^k/\R(1,1,\dots,1)\,$ is the \emph{tropical torus}, and
the tropical projective space $\TP^{k-1}$ is a compactification of
$\TA^{k-1}$ which is a closed simplex.

The symmetric group $\Sym_7$ acts naturally on both $\Gr(3,7)$ and
$\Dr(3,7)$, and it makes sense to count their cells up to this
symmetry.  The face numbers of the underlying polytopal complexes
modulo $\Sym_7$ are
\begin{eqnarray*}
f(\Gr(3,7) \bmod \Sym_7) &  = & (6,37,140,296,300,125)\quad \text{and} \\
f(\Dr(3,7) \bmod \Sym_7) &  = & (5,30,107,217,218,94,1)\;.
\end{eqnarray*}
Thus the Grassmannian $\Gr(3,7)$ modulo $\Sym_7$ has $125 $
five-dimensional simplices, and these are merged to $94$
five-dimensional polytopes in the Dressian $\Dr(3,7)$ modulo $\Sym_7$.
One of these cells is not a facet because it lies in the unique cell
of dimension six. This means that $\Dr(3,7)$ has $93 + 1 = 94$ facets
(= maximal cells) up to the $\Sym_7$-symmetry.

Each point in $\Dr(3,n)$ determines a plane in $\TP^{n-1}$.  This map
was described in \cite{Spe1, SS} and we recall it in Section
\ref{sec:pictures}.  The cells of $\Dr(3,n)$ modulo $\Sym_n$
correspond to combinatorial types of tropical planes.  Facets of
$\Dr(3,n)$ correspond to \emph{generic planes} in $\TP^{n-1}$:

\begin{cor}
  The number of combinatorial types of generic planes in $\TP^6$ is
  $94$. The numbers of types of generic planes in
  $\TP^3$, $\TP^4$, and $\TP^5$ are $1$, $1$, and $7$, respectively.
\end{cor}

\begin{proof} The unique generic plane in $\TP^3$ is the cone over the
  complete graph $K_4$. Planes in $\TP^4$ are parameterized by the
  Petersen graph $\,\Dr(3,5) = \Gr(3,5)$, and the unique generic type
  is dual to the trivalent tree with five leaves.  The seven types of
  generic planes in $\TP^5$ were derived in \cite[\S 5]{SS}. Drawings
  of their bounded parts are given in Figure~\ref{fig:planes}, while
  their unbounded cells are represented by the tree arrangements in
  Table~\ref{tab:trees} below.  The number $94$ for $n = 7$ is derived
  from Theorem \ref{thm:Dr37}.
\end{proof}

A complete census of all combinatorial types of tropical planes in
$\TP^6$ is posted at
\begin{quote}\small
  \url{www.uni-math.gwdg.de/jensen/Research/G3_7/grassmann3_7.html}.
\end{quote}
This web site and the notation used therein is a main contribution of
the present paper.  In the rest of this section we explain how our two
classification theorems were obtained.

\medskip

\begin{proof}[Computational proof of Theorem \ref{thm:Gr37}]
  The Grassmannian $\Gr(3,7)$ is the tropical variety defined by the
  Pl\"ucker ideal $I_{3,7}$ in the polynomial ring $\K[p_S]$ in $35$
  unknowns. We first suppose that $\K$ has characteristic zero, and
  for our computations we take $\K = \Q$.  The subvariety of
  $\PP^{34}_\Q$ defined by $I_{3,7}$ is irreducible of dimension $12$ and
  has an effective six-dimensional torus action. The Bieri-Groves
  Theorem \cite{BG} ensures that $\Gr(3,7)$ is a pure five-dimensional
  subcomplex of the Gr\"obner complex of $I_{3,7}$.  Moreover, by
  \cite[Theorem~3.1]{BJSST}, this complex is connected in codimension
  one.  The software \Gfan \cite{Gfan} exploits this connectivity by
  traversing the facets exhaustively when computing $\Gr(3,7) =
  \tropical{I_{3,7}}$.

  The input to \Gfan is a single maximal Gr\"obner cone of the
  tropical variety. The cone is, as described in the \Gfan manual,
  represented by a pair of Gr\"obner bases. Knowing a relative
  interior point of a maximal cone we can compute this pair with the
  command
\begin{verbatim}
gfan_initialforms --ideal --pair
\end{verbatim} 
  run on the input
\begin{verbatim}
Q[p123,p124,p125,p126,p127,p134,p135,p136,p137,p145,p146,p147,
p156,p157,p167,p234,p235,p236,p237,p245,p246,p247,p256,p257,p267,
p345,p346,p347,p356,p357,p367,p456,p457,p467,p567]
{
p123*p145-p124*p135+p125*p134,
....
p123*p456-p124*p356+p125*p346+p126*p345,
....
p347*p567-p357*p467+p367*p457
}
( 0, 0, 0, 0, 0, 0, 0, 0, 0, -2, -3, -2, -2, -3, -2, 0, 0, 0, 0,
-3, -1, -2, -1, -2, -1, -2, -1, -3, -1, -2, -1, -3, -4, -3, -5).
\end{verbatim}
  The polynomials are the $140$ quadrics which minimally generate the
  Pl\"ucker ideal $I_{3,7}$. Among these are $105$ three-term
  relations and $35$ four-term relations.  Since \Gfan uses the
  max-convention for tropical addition, weight vectors have to be
   negated. The output
  is handed over to the program \texttt{gfan\_tropicaltraverse}, which
  computes all other maximal cones. For this computation to finish it
  is decisive to use the \texttt{--symmetry} option. The symmetric
  group $\Sym_7$ acts on the tropical Pl\"ucker coordinates as a
  subgroup of $\Sym_{35}$.  In terms of classical Pl\"ucker
  coordinates, these symmetries only exist if we simultaneously
  perform sign changes, such as $p_{132} = -p_{123}$. We inform \Gfan
  about these sign changes using \texttt{--symsigns}, and we specify the
  sign changes on the input as elements of $\{-1,+1\}^{35}$ together
  with the generators of $\Sym_7\subset \Sym_{35}$ after the Gr\"obner
  basis pair produced above:
\begin{verbatim}
{(15,16,17,18,0,19,20,21,1,22,23,2,24,3,4,25,26,27,5,28,29,6,30,7,8,31,
32,9,33,10,11,34,12,13,14),(0,1,2,3,4,15,16,17,18,19,20,21,22,23,24,5,
6,7,8,9,10,11,12,13,14,25,26,27,28,29,30,31,32,33,34)}
{(1,1,1,1,1,1,1,1,1,1,1,1,1,1,1,1,1,1,1,1,1,1,1,1,1,1,1,1,1,1,1,1,1,1,1),
(-1,-1,-1,-1,-1,1,1,1,1,1,1,1,1,1,1,1,1,1,1,1,1,1,1,1,1,1,1,1,1,1,1,1,1,
1,1)}
\end{verbatim}
  Before traversing $\Gr(3,7)$, \Gfan verifies algebraically that these
  indeed are symmetries.

  In order to handle a tropical variety as large as $\Gr(3,7)$, the
  implementation of the traversal algorithm in \cite{BJSST} was improved in
  several ways.  During the traversal of the maximal cones up to symmetry,
  algebraic tests were translated into polyhedral containment questions
  whenever possible. Since the fan turned out to be simplicial, computing the
  rays could be reduced to linear algebra while in general \Gfan uses the
  double description method of \cddlib~\cite{Cddlib}. In the subsequent
  combinatorial extraction of all faces up to symmetry, checking if two cones
  are in the same orbit can be done at the level of canonical interior points.
  Checking if two points are equal up to symmetry was done by running through
  all permutations in the group.  This may not be optimal but is sufficient
  for our purpose. For further speed-ups we linked \Gfan to the floating point
  LP solver \SoPlex \cite{Soplex} which produced certificates verifiable in
  integer arithmetic. In case of a failure caused by round-off errors, the
  program falls back on \cddlib which solves the LP problem in exact
  arithmetic. The running time for the computation is approximately 25 hours
  on a standard desktop computer with \Gfan version~0.4, which will be
  released by May 2009.  The output of \Gfan is in \polymake \cite{polymake}
  format, and the program \homology~\cite{homology} was used to compute the
  integral homology of the underlying polytopal complex.

  The above computations established our result in characteristic
  zero.  To obtain the same result for prime characteristics $p \geq
  3$, we used \Macaulay to redo all Gr\"obner basis
  computations, one for each cone in $\Gr(3,7)$, in the polynomial
  ring $\Z[p_S]$ over the integers. We found that all but one of the
  initial ideals $\initial_\lambda(I_{3,6})$ arise from $I_{3,6}$ via
  a Gr\"obner basis whose coefficients are $+1$ and $-1$. Hence
  these cones of $\Gr(3,7)$ are characteristic-free. The only
  exception is the Fano cone which will be discussed in the
  end of Section~\ref{sec:matroids}.
\end{proof}

\begin{proof}[Computational proof of Theorem \ref{thm:Dr37}]
  For $d=3$ and $n=7$ there are $105$ three-term Pl\"ucker relations
  (\ref{eq:3term}). A vector $\lambda \in \R^{35}$ lies in $\Dr(3,7)$
  if and only if the initial form of each three-term relation with
  respect to $\lambda$ has either two or three terms. There are four
  possibilities for this to happen, and each choice is described by a
  linear system of equations and inequalities.  This system is
  feasible if and only if the corresponding cone exists in the
  Dressian $\Dr(3,7)$, and this can be tested using linear
  programming.  In theory, we could compute the Dressian by running
  a loop over all $4^{105}$ choices and list which choices determine a
  non-empty cone of $\Dr(3,7)$.  Clearly, this is infeasible in
  practice.

  To control the combinatorial explosion, we employed the
  representation of tropical planes by abstract tree arrangements
  which will be introduced in Section 4. This representation allows a
  recursive computation of $\Dr(3,n)$ from $\Dr(3,n-1)$. The idea is
  similar to what is described in the previous paragraph, but the
  approach is much more efficient.  By taking the action of the
  symmetric group of degree $n$ into account and by organizing this
  exhaustive search well enough this leads to a viable computation.  A
  key issue seems to be to focus on the equations early in the
  enumeration, while the inequalities are considered only at the very
  end.  A \polymake implementation enumerates all cones of $\Dr(3,7)$
  within one hour.  The same computation for $\Dr(3,6)$ takes less
  than two minutes.

  Again we used \homology for computing the integral homology of the
  underlying polytopal complex of $\Dr(3,7)$.  Since $\Dr(3,7)$ is not
  simplicial it cannot be fed into $\homology$ directly.  However, it
  is homotopy equivalent to its crosscut complex, which thus has the
  same homology \cite{Bjor}.  The \emph{crosscut complex} (with
  respect to the atoms) is the abstract simplicial complex whose
  vertices are the rays of $\Dr(3,7)$ and whose faces are the subsets
  of rays which are contained in cones of $\Dr(3,7)$.  The computation
  of the homology of the crosscut complex takes about two hours.
\end{proof}

\begin{remark} \label{rem:DW} Following \cite{DW1,DW2}, a
  \emph{valuated matroid} of rank~$d$ on the set $[n]$ is a map
  $\pi:[n]^d\to\R\cup\{\infty\}$ such that $\pi(\omega)$ is
  independent of the ordering of the sequence $\omega$,
  $\,\pi(\omega)=\infty$ if an element occurs twice in $\omega$, and
  the following axiom holds: for every $(d - 1)$-subset $\sigma$ and
  every $(d + 1)$-subset $\tau=\{\tau_1,\tau_2,\ldots,\tau_{d+1}\}$ of
  $[n]$ the minimum of
  \[
  \pi(\sigma\cup\{\tau_i\})+\pi(\tau\backslash\{\tau_i\})
  \quad \hbox{for} \quad 1\le i\le d+1
  \]
  is attained at least twice.  Results of Dress and Wenzel \cite{DW1} 
  imply that
  tropical Pl\"ucker vectors and valuated matroids are the same.  To
  see this, one applies \cite[Theorem~3.4]{DW1} to the perfect
  fuzzy ring arising from $(\R\cup\{\infty\},\min,+)$ via the
  construction in \cite[page~182]{DW1}.
\end{remark}

\section{Matroid Subdivisions}\label{sec:matroids}

A \emph{weight function} $\lambda$ on an $n$-dimensional polytope $P$
in $\R^n$ assigns a real number to each vertex of $P$.  The lower
facets of the lifted polytope $\conv\smallSetOf{(v,\lambda(v))}{v
  \text{ vertex of } P}$ in $\R^{n+1}$ induce a polytopal subdivision
of $P$.  Polytopal subdivisions arising in this way are called
\emph{regular}. The set of all weights inducing a fixed subdivision
forms a (relatively open) polyhedral cone, and the set of all these
cones is a complete fan, the \emph{secondary fan} of $P$.  The
dimension of the secondary fan as a spherical complex is $m-n-1$,
where $m$ is the number of vertices of~$P$. For a detailed
introduction to these concepts see~\cite{DRS}.

We denote the canonical basis vectors of $\R^n$ by
$e_1,e_2,\dots,e_n$, and we abbreviate $e_X:=\sum_{i\in X}e_i$ for any
subset $X\subseteq[n]$. For a set $\cX\subseteq\tbinom{[n]}{d}$ we
define the polytope
\[
\MatroidPolytope{\cX} \ := \ \conv\SetOf{e_X}{X\in\cX} \, .
\]
The $d$-th
\emph{hypersimplex} in $\R^n$ is the special case
\[
\Delta(d,n) \ := \ \MatroidPolytope{\tbinom{[n]}{d}} \, .
\]
A subset $\cM\subseteq\tbinom{[n]}{d}$ is a \emph{matroid} of
\emph{rank} $d$ on the set $[n]$ if the edges of the polytope
$\MatroidPolytope{\cM}$ are all parallel to the edges of
$\Delta(d,n)$; in this case $\MatroidPolytope{\cM}$ is called a
\emph{matroid polytope}, and the elements of $\cM$ are the
\emph{bases}.  That this definition really describes a matroid as, for
example, in White \cite{White86}, is a result of Gel{\cprime}fand,
Goresky, MacPherson, and Serganova~\cite{GGMS87}.  Moreover, each face
of a matroid polytope is again a matroid polytope \cite{FS}.  A polytopal
subdivision of $\Delta(d,n)$ is a \emph{matroid subdivision} if
each of its cells is a matroid polytope.

\begin{prop}{\rm (Speyer \cite[Proposition~2.2]{Spe1})}
  \label{prop:Speyer} A weight vector $\lambda \in
  \R^{\tbinom{[n]}{d}}$ lies in the Dressian $\Dr(d,n)$, seen as a
  fan, if and only if it induces a matroid subdivision of the
  hypersimplex $\Delta(d,n)$.
\end{prop}

The weight functions inducing matroid subdivisions form a
subfan of the secondary fan of $\Delta(d,n)$, and this defines the
\emph{secondary fan structure} on the Dressian $\Dr(d,n)$.  It is not
obvious whether the secondary fan structure and the Pl\"ucker fan
structure on $\Dr(d,n)$ coincide. We shall see in
Theorem~\ref{thm:arrangement} that this is indeed the case
for $d=3$.  In particular, the rays of the Dressian
$\Dr(3,n)$ correspond to coarsest matroid
subdivisions of $\Delta(3,n)$.

\begin{cor}\label{cor:matroid}
  Let $M$ be a connected matroid of rank $d$ on $[n]$ and let
  $\lambda_M \in \{0,1\}^{\tbinom{[n]}{d}}$ be the vector which
  satisfies $\lambda_M(X) = 0$ if $X$ is a basis of $M$ and
  $\lambda_M(X) = 1$ if $X$ is not a basis of $M$.  Then $\lambda_M$
  lies in the Dressian $\Dr(d,n)$, and the corresponding matroid
  decomposition of $\Delta(d,n)$ has the matroid polytope
  $\MatroidPolytope{M}$ as a maximal cell.
\end{cor}

\begin{proof}
  The basis exchange axiom for matroids translates into a
  combinatorial version of the quadratic Pl\"ucker relations
  (cf.~Remark \ref{rem:DW}), and this ensures that the vector
  $\lambda_M$ lies in the Dressian $\Dr(d,n)$.  By Proposition
  \ref{prop:Speyer}, the regular subdivision of $\Delta(d,n)$ defined
  by $\lambda_M$ is a matroid subdivision.  The matroid polytope
  $\MatroidPolytope{M}$ appears as a lower face in the lifting of
  $\Delta(d,n)$ by $\lambda_M$, and hence it is a cell of the matroid
  subdivision.  It is a maximal cell because ${\rm
    dim}(\MatroidPolytope{M}) = n-1$ if and only if the matroid $M$ is
  connected; see \cite{FS}.
\end{proof}

Each vertex figure of $\Delta(d,n)$ is isomorphic to the product of
simplices $\Delta_{d-1}\times\Delta_{n-d-1}$.  A regular subdivision
of a polytope induces regular subdivisions on its facets as well as on
its vertex figures.  For hypersimplices the converse holds
(see also Proposition~\ref{prop:KK}):

\begin{prop}{\rm (Kapranov \cite[Corollary~1.4.14]{K})}\label{prop:K}.
  Each regular subdivision of 
the product of simplices $\Delta_{d-1}\times\Delta_{n-d-1}$ is
  induced by a regular matroid subdivision of $\Delta(d,n)$.
\end{prop}

A \emph{split} of a polytope is a regular subdivision with exactly two
maximal cells.  By \cite[Lemma~7.4]{HJ}, every split of $\Delta(d,n)$ is a matroid subdivision.
Collections of splits that are pairwise compatible define a
simplicial complex, known as the \emph{split complex} of $\Delta(d,n)$.
It was shown in \cite[Section~7]{HJ} that the regular subdivision defined
by pairwise compatible splits is always a matroid subdivision. The
following result appears in \cite[Theorem~7.8]{HJ}:

\begin{prop}
  \label{prop:HJ}
  The split complex of $\Delta(d,n)$ is a simplicial subcomplex
  of the Dressian $\Dr(d,n)$, with its secondary complex structure.
They are equal if $d=2$ or $d=n-2$.
\end{prop}

Special examples of splits come about in the following way.  The
vertices adjacent to a fixed vertex of $\Delta(d,n)$ span a hyperplane
which defines a split; and these splits are called \emph{vertex
  splits}.  Moreover, two vertex splits are compatible if and only if
the corresponding vertices of $\Delta(d,n)$ are not connected by an
edge. Hence the simplicial complex of stable sets of the edge graph of
$\Delta(d,n)$ is contained in the split complex of $\Delta(d,n)$.

\begin{cor}
  \label{cor:dimlowbd}
  The simplicial complex of stable sets of the edge graph of the
  hypersimplex $\Delta(d,n)$ is a subcomplex of $\Dr(d,n)$.  Hence,
  the dimension of the Dressian $\Dr(d,n)$, seen as a polytopal
  complex, is bounded below by one less than the maximal size of a
  stable set of this edge graph.
\end{cor}

We shall use this corollary to prove the main result in this section.
Recall that the dimension of the Grassmannian $\Gr(3,n)$ equals
$2n-9$.  Consequently, the following theorem implies that, for large
$n$, most of the tropical planes (cf.~Section \ref{sec:pictures}) are
not realizable.

\begin{thm}\label{thm:dim_bound}
  The dimension of  the Dressian $\Dr(3,n)$ is of order $\Theta(n^2)$.
\end{thm}

For the proof of this result we need one more definition.  The
\emph{spread} of a vector in $\Dr(d,n)$ is the number of maximal cells
of the corresponding matroid decomposition.  The splits are precisely
the vectors of spread~2, and these are rays of $\Dr(d,n)$, seen as a
fan.  The rays of $\Dr(3,6)$ are either of spread 2 or~3;
see~\cite[\S~5]{SS}.  As a result of our computation the spreads of
rays of $\Dr(3,7)$ turn out to be $2$, $3$, and $4$. We note the
following result.

\begin{prop}
  As $n$ increases, the spread of the rays of $\Dr(3,n)$ is not
  bounded by a constant.
\end{prop}

\begin{proof}
  By Proposition~\ref{prop:K}, each regular subdivision of
  $\Delta_{2}\times\Delta_{n-4}$ is induced by a regular matroid
  subdivision of $\Delta(3,n)$, and hence, in light of the Cayley
  trick \cite{S}, by mixed subdivisions of the dilated triangle
  $(n-3)\Delta_{2}$.  See also Section \ref{sec:trees}.  This
  correspondence maps rays of the secondary fan of
  $\Delta_{2}\times\Delta_{n-4}$ to rays of the Dressian $\Dr(3,n)$.
  Now, a coarsest mixed subdivision of $(n-3)\Delta_2$ can have
  arbitrarily many polygons as $n$ grows large. For an example
  consider the hexagonal subdivision in \cite[Figure 12]{S}.  Hence a
  coarsest regular matroid subdivision of $\Delta(3,n)$ can have
  arbitrarily many facets.
\end{proof}

\begin{proof}[Proof of Theorem \ref{thm:dim_bound}]
  Speyer \cite[Theorem~6.1]{Spe1} showed that the spread of any vector
  in $\Dr(d,n)$ is at most $\tbinom{n-2}{d-1}$.  This is the maximal
  number of facets of any matroid subdivision of $\Delta(d,n)$.
  Consider a flag of faces $F_1\subset F_2\subset\cdots$ in the
  underlying polytopal complex of $\Dr(d,n)$.  For every $i$ the
  subdivision corresponding to $F_i$ has more facets than that of
  $F_{i-1}$.  Hence $\tbinom{n-2}{d-1}-1$ is an upper bound for the
  dimension of $\Dr(d,n)$.  Specializing to $d=3$, this upper bound is
  quadratic.

  We shall apply Proposition~\ref{prop:HJ} to derive the lower bound.  The
  \emph{generalized Fano matroid} $\cF_r$ is a connected simple matroid on
  $2^r-1$ points which has rank $3$ and is defined as follows.  Its
  three-element circuits are the lines of the $(r-1)$-dimensional projective
  space $\PG_{r-1}(2)$ over the field $\GF(2)$ with two elements.  The total
  number of unordered bases of $\cF_r$, that is, non-collinear triples of
  points, equals
  \[
  \beta_r \ := \ \frac{1}{6}(2^r-1)(2^{r}-2)(2^r-4) \, .
  \]
  The number of vertices of $\Delta(3,2^r-1)$ which are not bases of
  $\cF_r$ equals
  \[
  \nu_r \ := \ \binom{2^r-1}{3}-\beta_r \ = \
  \frac{1}{6}(2^r-1)(2^r-2) \, .
  \]
  We claim that the non-bases of $\cF_r$ form a stable set in the edge
  graph of $\Delta(3,2^r-1)$. Indeed, the non-bases are precisely the
  collinear triplets of points, that is, the full point rows of the
  lines in $\PG_{r-1}(2)$.  Two distinct point rows of lines in
  $\PG_{r-1}(2)$ share at most one point, and hence the two
  corresponding vertices of $\Delta(3,2^r-1)$ do not differ by an
  exchange of two bits, which means that they are not connected by an
  edge.

  The quadratic lower bound is now derived from
  Proposition~\ref{prop:HJ} as follows. For given any $n$, let $r$ be the
  unique natural number satisfying $2^r-1 \le n < 2^{r+1}$.  Then the
  generalized Fano matroid $\cF_r$ yields a stable set of size
  $\,\nu_r={1}/{6}(2^r-1)(2^r-2)\ge n^2/24-n/12\,$ in the edge graph
  of $\Delta(3,n)$.  The latter inequality follows from $\,2^r-1\ge
  n/2\,$.
\end{proof}

\begin{figure}[htb]
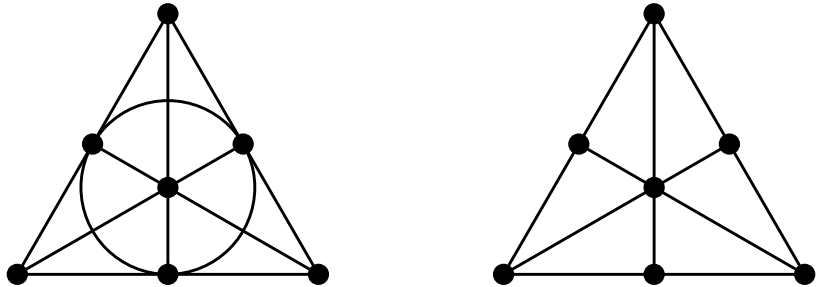
\centering
  \includegraphics[scale=1.156]{non-Fano.1}
  \hskip 2cm
  \includegraphics[scale=1.156]{non-Fano.0}
  \caption{The point configurations for the Fano and
    non-Fano matroids.}
  \label{fig:non-fano}
\end{figure}

\begin{proof}[Computational proof of Theorem \ref{thm:Gr37} (continued)]
  We still have to discuss the Fano cone of $\Dr(3,7)$ and its
  relationship to $\Gr(3,7)$.  The matroid $\cF_3$ in the proof of
  Theorem \ref{thm:dim_bound} corresponds to the Fano plane $\PG_2(2)$,
  which is
  shown in Figure~\ref{fig:non-fano} on the left. We have $\beta_3=28$
  and $\nu_3=7$.  Via Corollary~\ref{cor:matroid} the Fano matroid
  $\cF_3$ gives rise to a cone in the fan $\Dr(3,7)$ which we call the
  \emph{Fano cone}.  The corresponding cell of $\Dr(3,7)$, seen as a
  polytopal complex, has dimension six.  Moreover, all $30$
  six-dimensional cells of $\Dr(3,7)$ come from the Fano matroid
  $\cF_3$ by relabeling. They form a single orbit under the $\Sym_7$
  action, since the automorphism group $\GL_3(2)$ of $\cF_3$ has order
  $168=5040/30$.  If the field $\K$ considered has characteristic $2$
  then the Fano cell of $\Dr(3,7)$ intersects $\Gr(3,7)$ in a
  five-dimensional complex that looks like a tropical hyperplane.

  Finally, suppose that the characteristic of $\K$ is different from
  $2$. Since the Fano matroid is not realizable over $\K$, the Fano cone
  of $\Dr(3,7)$ corresponds to a non-realizable tropical plane in $\TP^6$
  and the intersection of the Fano cell with $\Gr(3,7)$ is a
  five-dimensional simplicial sphere arising from seven copies of the
  non-Fano matroid; see Figure~\ref{fig:non-fano} on the right.  In
  this case this also gives us the difference in the homology of
  $\Dr(3,7)$ and $\Gr(3,7)$.  The Fano six-cells are simplices.
 Each of them cancels precisely one homology cycle of $\Gr(3,7)$. 
\end{proof}

In spite of the results in this sections, many open problems remain.
Here are two specific questions we have concerning the combinatorial
structure of the Dressian $\Dr(3,n)$:

\begin{itemize}
\item Are all rays of $\Dr(3,n)$ always rays of $\Gr(3,n)$?
\item Characterize the rays of $\Dr(3,n)$, that is, coarsest matroid
  subdivisions of $\Delta(3,n)$.
\end{itemize}

\section{Tree Arrangements}\label{sec:trees}

Let $n \ge 4$ and consider an $n$-tuple of metric trees
$T=(T_1,T_2,\dots,T_n)$ where $T_i$ has the set of leaves
$[n]\backslash\{i\}$.  A \emph{metric tree} $T_i$ by definition comes
with non-negative edge lengths, and by adding lengths along paths it
defines a metric
$\delta_i:([n]\backslash\{i\})\times([n]\backslash\{i\})\to\R_{\ge
  0}$.  We call the $n$-tuple $T$ of metric trees a \emph{metric tree
  arrangement} if
\begin{equation}\label{eq:metric_tree}
  \delta_i(j,k) \ = \ \delta_j(k,i) \ = \ \delta_k(i,j)
\end{equation}
for all $i,j,k\in[n]$ pairwise distinct.  Moreover, considering 
trees $T_i$ without metrics, but with leaves
still labeled by $[n]\backslash \{i\}$,
we say that $T$ 
is an \emph{abstract tree arrangement} if 
\begin{itemize}
\item either $n=4$;
\item or $n=5$, and $T$ is the set of quartets of a tree with five leaves;
\item or $n \ge 6$, and $(T_1\backslash i,\dots,T_{i-1}\backslash
  i,T_{i+1}\backslash i,\dots,T_n\backslash i)$ is an arrangement of
  $n-1$ trees for each $i\in[n]$.
\end{itemize}
Here $T_j\backslash i$ denotes the tree on $[n]\backslash\{i,j\}$
gotten by deleting leaf $i$ from tree $T_j$. 
A \emph{quartet} of a tree is a subtree induced by four of its leaves.

The following result relates the two concepts of tree arrangements we
introduced:

\begin{prop}
  Each metric tree arrangement gives rise to an abstract tree
  arrangement by ignoring the edge lengths. The converse is not true:
  for $n \geq 9$, there exist abstract arrangements of $n$ trees that
  do not support any metric tree arrangement.
\end{prop}

\begin{proof} The first assertion follows from the Four Point Condition; see
  \cite[Theorem 2.36]{ASCB}. An example establishing the second assertion is
  the abstract arrangement of nine trees listed in Table \ref{tab:paco} and
  depicted in Figure \ref{fig:paco}: Three of the trees (numbered $1,2,3$) are
  on the boundary, while the six remaining trees (numbered $4,5,6,7,8,9$)
  partition the dual graph of the subdivision of the big triangle into
  quadrangles and small triangles.  Each intersection of the tree $T_a$ with
  the tree $T_b$ in one of the quadrangles defines a leaf labeled $b$ in $T_a$
  and, symmetrically, a leaf labeled $a$ in $T_b$.  See Example \ref{exa:paco}
  below for more details, including an argument why this abstract arrangement
  cannot be realized as a metric arrangement.
\end{proof}

\begin{figure}[hbt]\centering
  \includegraphics[scale=0.95]{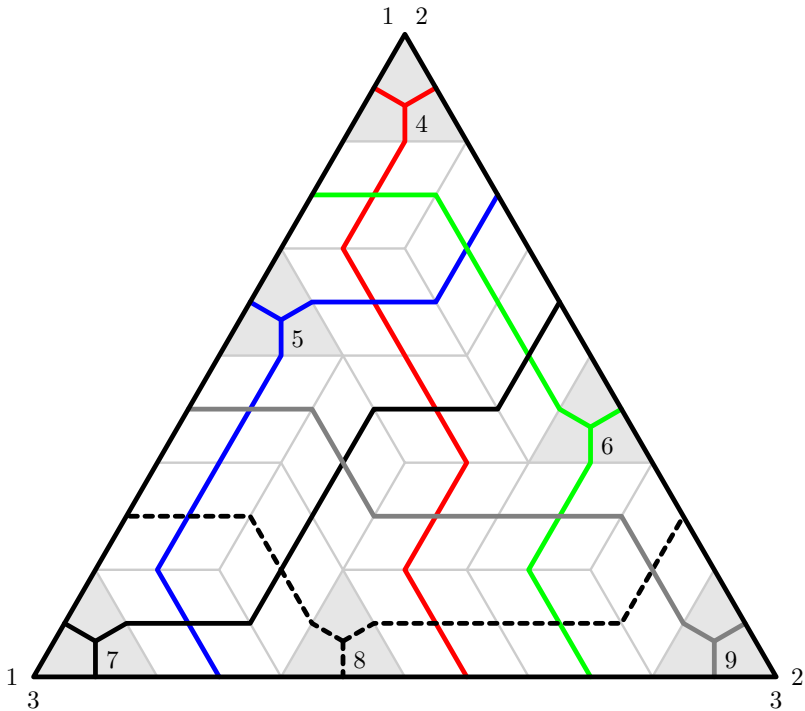}
  \caption{Abstract arrangement of nine caterpillar trees on eight
    leaves encoding a matroid subdivision of $\Delta(3,9)$ which is
    not regular; see Table~\ref{tab:paco}.}
  \label{fig:paco}
\end{figure}

The hypersimplex $\Delta(d,n)$ is the intersection of the unit cube
$[0,1]^n$ with the affine hyperplane $\sum x_i=d$.  From this it
follows that the facets of $\Delta(d,n)$ correspond to the facets of
$[0,1]^n$.  We call the facet defined by $x_i=0$ the \emph{$i$-th
  deletion facet} of $\Delta(d,n)$, and the facet defined by $x_i=1$
the \emph{$i$-th contraction facet}.  These names come about as
follows: If $\cM$ is a matroid on $[n]$ of rank $d$ then the
intersection of $\MatroidPolytope{\cM}$ with the $i$-th deletion
(contraction) facet is the matroid polytope of the matroid obtained by
deleting (contracting)~$i$.  Each deletion facet of $\Delta(d,n)$ is
isomorphic to $\Delta(d,n-1)$, and each contraction facet is
isomorphic to $\Delta(d-1,n-1)$.  We use the terms ``deletion'' and
``contraction'' also for matroid subdivisions and for vectors in
$\R^{\tbinom{[n]}{d}}$.  Notice that trees come naturally into the
game since a polytopal subdivision of $\Delta(2,n-1)$ is a matroid
subdivision if and only if it is dual to a tree.

\begin{lem}\label{lem:arrangement}
  Each matroid subdivision $\Sigma$ of $\Delta(3,n)$ defines
  an abstract arrangement $T(\Sigma)$ of $n$ trees.  Moreover, if
  $\Sigma$ is regular then $T(\Sigma)$ supports a metric
tree arrangement.
\end{lem}

\begin{proof}
  Each of the $n$ contraction facets of $\Delta(3,n)$ is isomorphic to
  $\Delta(2,n-1)$, and thus $\Sigma$ induces matroid subdivisions on
  $n$ copies of $\Delta(2,n-1)$.  But the matroid subdivisions of
  $\Delta(2,n-1)$ are generated by compatible systems of splits, see
  \cite[Theorem 3.4]{SS} and \cite[\S6]{HJ}.  These matroid
  subdivisions are dual to trees, and hence $\Sigma$ gives rise to a
  tree arrangement.

  Now let $\Sigma$ be regular with weight function $\lambda$. By
  adding or subtracting a suitable multiple of $(1,1,\ldots,1)$ and
  subsequent rescaling we can assume that $\lambda$ attains values
  between $1$ and $2$ only.  The function $\lambda$ can be restricted
  to a weight function on each contraction facet.  But a weight
  function on $\Delta(2,n-1)$ which takes values between $1$ and $2$
  is a metric.  Since the induced regular subdivisions of the facets
  of $\Delta(3,n)$ isomorphic to $\Delta(2,n-1)$ are also regular
  matroid subdivisions, they are dual to metric trees with $n-1$
  leaves.  The Split Decomposition Theorem of Bandelt and
  Dress~\cite[Theorem~2]{BD} allows to read off the lengths on all
  edges of these trees; see also \cite[Theorem~3.10]{HJ}.
\end{proof}

\begin{prop}\label{prop:bar_sigma}
  Let $\Sigma$ and $\bar \Sigma$ be two matroid subdivisions of
  $\Delta(3,n)$ such that $\Sigma$ refines $\bar\Sigma$.  If $\Sigma$
  and $\bar\Sigma$ induce the same subdivision on the boundary of
  $\Delta(3,n)$ then $\Sigma$ and $\bar\Sigma$ coincide.
\end{prop}

\begin{proof}
  Suppose that $\Sigma$ strictly refines $\bar\Sigma$.  Then there is
  a codimension-$1$-cell $F$ of $\Sigma$ which is not a cell in
  $\bar\Sigma$.  Let $\bar F$ be the unique full-dimensional cell of
  $\bar \Sigma$ that contains $F$.  In particular, $F$ is not
  contained in the boundary of $\Delta(3,n)$.  Then $F$ is a
  rank-$3$-matroid polytope $F=\MatroidPolytope{\cM}$ of codimension
  $1$ where $\cM=\cM_1\cup\cM_2$ is the disjoint union of a
  rank-$1$-matroid $\cM_1$ and a rank-$2$-matroid $\cM_2$.  In
  particular,
  $F\cong\MatroidPolytope{\cM_1}\times\MatroidPolytope{\cM_2}$.
  Notice that the affine hull $H$ of $F$ is defined by the equation
  $\sum_{i\in I} x_i=1$ where we denote by $I$ the set of elements of
  $\cM_1$, all of which are parallel because of $\rank\cM_1=1$.

  Since $\bar F$ is subdivided by $H$ there exist vertices $v,w$ of
  $\bar F$ on either side of $H$.  Now $\bar F$ is also a matroid
  polytope of some matroid $\bar\cM$ containing $\cM$ as a submatroid.
  Up to relabeling we can assume that $v=e_{12i}$ and $w=e_{345}$ such
  that $\{1,2,i\}$ and $\{3,4,5\}$ are bases of $\bar\cM$ which are
  not bases of $\cM$ and where $1,2\in I$ and $i,3,4,5\not \in I$.  If
  $i\notin\{3,4,5\}$ we can exchange $i$ in the basis $\{1,2,i\}$ by
  some $j\in\{3,4,5\}$ to obtain a new basis of $\bar\cM$.  Without
  loss of generality we can assume that $i=5$ or $j=5$.  Hence
  $\{1,2,5\}$ and $\{3,4,5\}$ are bases of $\bar\cM$ that are not
  bases of $\cM$.  Notice that $e_{125}$ and $e_{345}$ are still on
  different sides of $H$ as $e_{12i}$ and $e_{125}$ are connected by
  an edge and $\{1,2,5\}$ is not a basis of $\cM$.
  
  Now, as $\rank\cM_i \le 2$, both $\cM_1$ and $\cM_2$ are realizable
  as affine point configurations over~$\R$, and hence so is $\cM$.  In
  the sequel we identify $\cM$ with a suitable point configuration
  (with multiple points).  This way we obtain a subconfiguration of
  five points in $\cM$ which looks like one of the two configurations
  shown in Figure~\ref{fig:matroid}.
  
  \begin{figure}[htb]
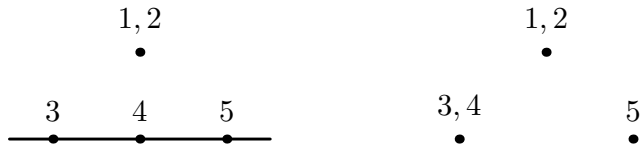
\centering
    \includegraphics[scale=1.156]{matroid.0}
    \hskip 2cm
    \includegraphics[scale=1.156]{matroid.1}
    \caption{Two point configurations in the Euclidean plane.}
    \label{fig:matroid}
  \end{figure}

  Consider the intersection of $\Delta(3,n)$ with the affine space
  defined by $x_5=1$ and $x_6=x_7=\dots=x_n=0$.  This gives us an
  octahedron $C\cong\Delta(2,4)$ in the boundary of $\Delta(3,n)$.
  The intersection $S=F\cap C$ is a square; it can be read off
  Figure~\ref{fig:matroid} as the convex hull of the four points
  $e_{135}$, $e_{145}$, $e_{235}$, and $e_{245}$. In particular, the
  square $S$ is a cell of $\Sigma$. However, since $e_{125}$ and
  $e_{345}$ are vertices of $\bar F=\MatroidPolytope{\bar\cM}$ as
  discussed above, $C$ is a cell of $\bar\Sigma$.  We conclude that
  the square $S$ is a cell of $\Sigma$ but not a cell of $\bar\Sigma$.
  By construction $S\subset C$ is contained in the boundary of
  $\Delta(3,n)$.  This yields the desired contradiction, as $\Sigma$
  and $\bar \Sigma$ induce the same subdivision on the boundary.
\end{proof}

Two metric tree arrangements are \emph{equivalent} if they induce the
same abstract tree arrangement.  The following is the main result of
this section.

\begin{thm}\label{thm:arrangement}
  The equivalence classes of arrangements of $n$ metric trees are in
  bijection with the regular matroid subdivisions of the hypersimplex
  $\Delta(3,n)$.  Moreover, the secondary fan structure on $\Dr(3,n)$
  coincides with the Pl\"ucker fan structure.
\end{thm}

\begin{proof}
  Each regular matroid subdivision defines a metric tree arrangement
  by Lemma~\ref{lem:arrangement}. The harder direction is to show that
  each metric tree arrangement gives rise to a regular matroid
  subdivision.  We will prove this by induction on~$n$.  The
  hypersimplex $\Delta(3,4)$ is a $3$-simplex without any non-trivial
  subdivisions, and $\Dr(3,4)$ is a single point corresponding to the
  unique equivalence class of metric trees.  The hypersimplex
  $\Delta(3,5)$ is isomorphic to $\Delta(2,5)$, and
  $\Dr(3,5) = \Gr(3,5)\cong\Gr(2,5)$ is isomorphic to the Petersen
  graph (considered as a one-dimensional polytopal complex).  Also in
  this case the result can be verified directly. This establishes the
  basis of our induction, and we now assume $n\ge 6$.

  Let $T$ be an arrangement of $n$ metric trees with tree metrics
  $\delta_1,\delta_2,\dots,\delta_n$.  In view of the axiom
  \eqref{eq:metric_tree}, the following map 
  $\pi:[n]^3\to\R\cup\{\infty\}$ is well-defined:
  \[
  \pi(i,j,k) \ = \
  \begin{cases}
    \delta_i(j,k)=\delta_j(k,i)=\delta_k(i,j) & \text{if $i,j,k$ are
      pairwise distinct,} \\
    \,\,\,\, \infty & \text{otherwise.}
  \end{cases}
  \]
  In order to show that $\pi$ is a tropical
  Pl\"ucker vector we have to verify that the minimum
  \[
  \min\bigl\{\pi_{hij}+\pi_{hkl}, \, \pi_{hik}+\pi_{hjl}, \, \pi_{hil}+\pi_{hjk}\bigr\}
  \]
  is attained at least twice, for any pairwise distinct
  $h,i,j,k,l\in[n]$.  Now, since $n\ge 6$, each $5$-tuple in $[n]$ is
  already contained in some deletion, and hence the desired property
  is satisfied by induction. We conclude that the restriction of the
  map $\pi$ to increasing triples $i<j<k$ is a finite tropical
  Pl\"ucker vector, that is, it is an element of $\Dr(3,n)$.  By
  Proposition~\ref{prop:Speyer} the map $\pi$ defines a matroid
  subdivision $\Sigma(T)$ of $\Delta(3,n)$.

  Consider any metric tree arrangement $T'$ that is equivalent to $T$.
  The maps $\pi$ and $\pi'$ associated with $T$ and $T'$ respectively
  clearly lie in the same cone of the Pl\"ucker fan structure on
  $\Dr(3,n)$. What we must prove is that they are also in the same
  cone of the secondary fan structure on $\Dr(3,n)$. Equivalently, we
  must show that $\Sigma(T')=\Sigma(T)$.  

  Suppose the secondary fan structure on $\Dr(3,n)$ is strictly finer
  than the Pl\"ucker fan structure.  Then there is a regular matroid
  subdivision $\Sigma$ of $\Delta(3,n)$ whose secondary cone
  $S(\Sigma)$ is strictly contained in the corresponding cone
  $P(\Sigma)$ of tropical Pl\"ucker vectors.  We fix a weight function
  $\lambda$ in the boundary of $S(\Sigma)$ which is contained in the
  interior of $P(\Sigma)$.  The matroid subdivision induced by
  $\lambda$ is denoted by $\bar\Sigma$.  By construction $\Sigma$
  strictly refines $\bar\Sigma$, and by induction we can assume that
  $\Sigma$ and $\bar\Sigma$ induce the same subdivision on the entire
  boundary of $\Delta(3,n)$.  Due to Proposition~\ref{prop:bar_sigma}
  we have that $\Sigma=\bar\Sigma$, and this completes our proof.
\end{proof}

We saw in Proposition~\ref{prop:K} that each regular subdivision of
$\Delta_{2}\times\Delta_{n-4}$ is induced by a regular matroid
subdivision of $\Delta(3,n)$. This implies that $\Dr(3,n)$ contains a
distinguished $(2n-9)$-dimensional sphere, dual to the secondary
polytope of $\Delta_{2}\times\Delta_{n-4}$, which parameterizes all
arrangements of $n-3$ lines in the tropical plane $\TP^2$. It has the
following nice description in terms of tree arrangements.  Let $L_1,
L_2, \ldots, L_{n-3}$ be the $n-3$ lines and let $L_x, L_y, L_z$
denote the three boundary lines of $\TP^2$.  Each of these $n$ lines
translates into a tree. The tree for $L_x$ is obtained by branching
off the leaves $\{1,2,\ldots,n-3\}$ on the path between leaves $y$ and
$z$, in the order in which the $L_j$ intersect $L_x$.  The trees for
$L_y$ and $L_z$ are analogous.  The tree for $L_i$ has one
distinguished node with long branches to the three special leaves $x$,
$y$ and $z$.  Along the path from the distinguished node to leaf $x$
we branch off additional leaves $j$ for each line $L_j$ that
intersects the line $L_i$ in its $x$-halfray.  This branching takes
place in the order in which the lines $L_j$ intersect $L_j$.  In this
manner, every arrangement of $n-3$ lines in $\TP^2$ translates into an
arrangement of $n$ trees.

The same construction also applies to arrangements of $n-3$ tropical
pseudolines in $\TP^2$ as defined by Ardila and Develin~\cite{AD}.  We
shall now describe this in terms of lozenge tilings as in \cite{S}.
Let $\Sigma$ be any polytopal subdivision of
$\Delta_{2}\times\Delta_{n-4}$.  The Cayley Trick encodes $\Sigma$ as
a mixed subdivision $M(\Sigma)$ of $(n-3)\Delta_{2}$, a regular
triangle of side length $n-3$.  By \cite[Theorem~3.5]{S} the mixed
subdivisions of dilated triangles are characterized as those polygonal
subdivisions whose cells are tiled by lozenges and upward triangles
(with unit edge lengths).  Here a \emph{lozenge} is a parallelogram
which is the union of one upward triangle and one downward triangle.
We call a mixed cell \emph{even} if it can be tiled by lozenges only.
Those which need an upward triangle in any tiling are \emph{odd}.  A
counting argument now reveals that each mixed subdivision of
$(n-3)\Delta_{2}$ contains up to $n-3$ odd polygonal cells.

\begin{prop}\label{prop:KK}
  Each polytopal subdivision $\Sigma$ of
  $\Delta_{2}\times\Delta_{n-4}$, or each mixed subdivision
  $M(\Sigma)$ of the triangle $(n-3)\Delta_{2}$, determines an
  abstract arrangement $T(\Sigma)$ of $n$ trees.
\end{prop}

\begin{proof}
  Assume that $\Sigma$ is a triangulation.  Equivalently, $M(\Sigma)$
  has exactly $n-3$ odd cells, all of which are upward triangles, and
  the even cells are lozenges.  Placing a labeled node into each
  upward triangle defines a tree in the dual graph of $\Sigma$.  Each
  of its three branches consists of the edges in $M(\Sigma)$ which are
  in the same parallel class as one fixed edge of that upward
  triangle.  Two opposite edges in a lozenge are parallel, and the
  \emph{parallelism} that we refer to is the transitive closure of
  this relation.  Each parallel class of edges extends to the boundary
  of the triangle $(n-3)\Delta_2$.  Doing so for all the upward
  triangles, we obtain an arrangement of tropical pseudo-lines
  \cite{AD}.  Each of these is subdivided by the intersection with the
  other tropical pseudo-lines.  We further add the three boundary
  lines of the big triangle to the arrangement.  This specifies an
  abstract tree arrangement $T(\Sigma)$.  Note that the trees in the
  arrangement partition the dual graph of the lozenge tiling.

  \begin{figure}[ht]
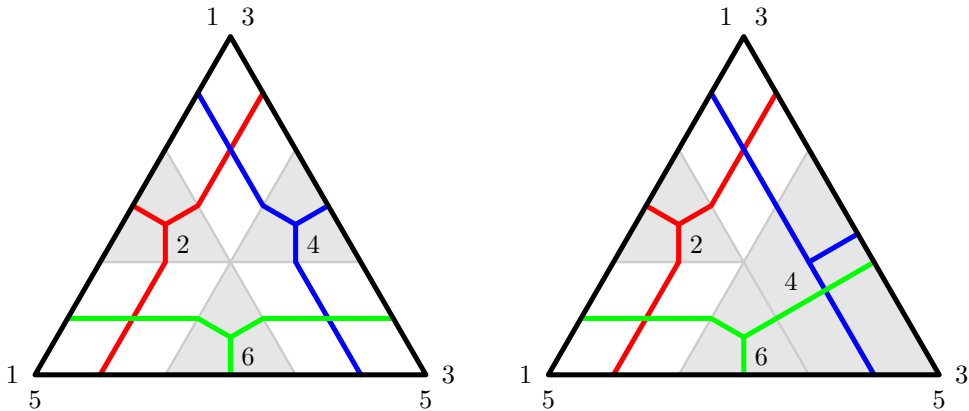
\centering
    \includegraphics[scale=1]{tree-arrangement.1}\qquad
    \includegraphics[scale=1]{tree-arrangement.2}
    \caption{Mixed subdivisions of $3\Delta_2$ and abstract
      arrangements of six trees.}
    \label{fig:lozenge}
  \end{figure}

  Now we consider the situation where $\Sigma$ is not a triangulation,
  so $M(\Sigma)$ is a coarser mixed subdivision of $(n-3)\Delta_2$.
  We shall associate a tree arrangement with $M(\Sigma)$.  Pick any
  triangulation $\Sigma'$ which refines $\Sigma$.  Then by the above
  procedure we have an abstract tree arrangement $T(\Sigma')$ induced
  by $\Sigma'$.  Then, as $\Sigma'$ refines $\Sigma$, one can contract
  edges in the trees of the arrangement $T(\Sigma')$.  In this way one
  also arrives at an abstract arrangement of $n$ trees.  Three of them
  correspond to the boundary lines of $(n-3)\Delta_2$.  The $n-3$
  non-boundary trees are assigned to the $ \leq n-3$ odd cells. Each
  cell is assigned at least one tree.  We note that $T(\Sigma)$ might
  depend on the choice of the triangulation $\Sigma'$.
\end{proof}

\begin{example}\label{ex:lozenge}
  Let $n=6$ and consider the two mixed subdivisions of $3\Delta_2$
  shown in Figure \ref{fig:lozenge}.  The left one is a lozenge tiling
  which encodes a regular triangulation of $\Delta_2\times\Delta_2$,
  here regarded as the vertex figure of $\Delta(3,6)$ at $e_{135}$.
  There are precisely three upward triangles, and each of them
  corresponds to a tree.  Moreover, the three sides of the big
  triangle encode three more trees. Using the notation of Figure
  \ref{fig:tree5}, this abstract tree arrangement equals
  \begin{equation}
    \label{eq:FFFGG}
    34\, 2\, 56 \, , \;\, 34\, 1\, 56 \, , \;\, 12\, 4\, 56 \, , 
    \;\, 12\, 3\,  56 \, , \;\, 12\, 6\, 34 \, , \;\, 12\, 5\, 34 \, .
  \end{equation}

  \begin{figure}[htb]\centering
    \includegraphics[scale=1]{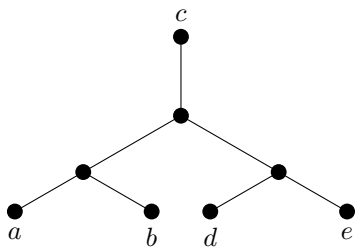}
    \caption{We use the notation $\,ab\, c\, de\,$ for this tree on
      five labeled leaves.}
    \label{fig:tree5}
  \end{figure}

  The tiling of $3\Delta_2$ on the right in Figure~\ref{fig:lozenge}
  is a mixed subdivision which coarsens the lozenge tiling discussed
  above.  It corresponds to the abstract tree arrangement
  \[
  34\, 2\, 56 \, , \; 34\, 1\, 56 \, , \; 12\, (456) \, , \; 12\,
  (356) \, , \; 12\, 6\, 34 \, , \; 12\, 5\, 34 \, .
  \]
  The tree $ab\, (cde)$ is obtained from the tree $ab\, c\, de$ by
  contracting the interior edge between $c$ and the pair $de$.  The
  odd polygonal cells (shaded in Figure~\ref{fig:lozenge}) correspond
  to trees. \qed
\end{example}

\begin{example} \label{exa:paco} An example of a non-regular matroid
  subdivision arises from the lozenge tiling of $6 \Delta_2$ borrowed
  from \cite{S} and shown in Figure~\ref{fig:paco}.  This picture
  translates into the abstract arrangement of nine trees in
  Table~\ref{tab:paco}.  The corresponding matroid subdivision of
  $\Delta(3,9) $ is not regular.  Equivalently, the Dressian
  $\Dr(3,9)$ has no cell for this tree arrangement.  \qed
\end{example}

\begin{table}[ht]\centering
  \caption{Abstract arrangement of nine caterpillar trees on eight
    leaves encoding a matroid subdivision of $\Delta(3,9)$ which is
    not regular; see Figure~\ref{fig:paco}.  The notation for
    caterpillar trees is explained in Figure~\ref{fig:trees6} below.}
  \renewcommand{\arraystretch}{0.9}
  \begin{tabular*}{.9\linewidth}{@{\extracolsep{\fill}}lll@{}}
    \toprule
    Tree 1: C$(24,6598,37)$ & Tree 2: C$(14,5768,39)$ & Tree 3: C$(17,5846,29)$ \\
    Tree 4: C$(12,6579,38)$ & Tree 5: C$(26,4198,37)$ & Tree 6: C$(14,5729,38)$ \\
    Tree 7: C$(13,5894,26)$ & Tree 8: C$(15,7346,29)$ & Tree 9: C$(15,7468,23)$ \\
    \bottomrule
  \end{tabular*}
  \label{tab:paco}
\end{table}

\begin{remark} 
  There are $187$ lozenge tilings of $4 \Delta_2$, each representing
  $24$ triangulations of $\Delta_3 \times \Delta_2$ via the $4! = 24$
  ways of labeling the upward triangles.  Each lozenge tiling defines
  an arrangement of seven trees that indexes a maximal cell of
  $\Dr(3,7)$.  In other words, the polytopal $5$-sphere dual to the
  secondary polytope of $\Delta_2 \times \Delta_3$ has $4488=187\cdot
  24$ facets, and embeds as a subcomplex into $\Dr(3,7)$.  It is
  instructive to study this subcomplex by browsing our website for
  $\Dr(3,7)$. For example, the tropical plane of type 89 on our
  website corresponds to Figure~4 in \cite{AD}.
\end{remark}

\begin{remark} 
  Another important sphere sitting inside the Grassmannian $\Gr(d,n)$,
  and hence in the Dressian $\Dr(d,n)$, is the \emph{positive
    Grassmannian} $\Gr^+(d,n)$, due to Speyer and Williams \cite{SW}.
  A natural next step would be to introduce and study the
  \emph{positive Dressian} $\Dr^+(d,n)$.  Generalizing \cite[\S
  5]{SW}, the positive Dressian $\Dr^+(3,n)$ would parameterize
  \emph{metric arrangements of planar trees}.  This space contains the
  $(2n-9)$-dimensional sphere $\Gr^+(3,n)$.  It would be interesting
  to know whether this inclusion is a homotopy equivalence, to explore
  relations with cluster algebras, and to extend the computation of
  $\Gr^+(3,7)$ presented in \cite{SW}. Incidentally, there is a
  misprint in the right part of \cite[Table 2]{SW}: the eleventh
  inequality should be ``$-x_5\leq -14$'' instead of ``$-17$''. With
  this correction, we independently verified the $f$-vector and the
  rays of its normal fan $F_{3,7}$.  \qed
\end{remark}

\section{Tropical Planes}
\label{sec:pictures}

We are now finally prepared to answer the question raised in the title
of this paper.  Tropical planes are contractible polyhedral surfaces
that are dual to the regular matroid subdivisions of $\Delta(3,n)$.
Consider any vector $p$ in $\R^{\binom{n}{3}}$ that lies in the
Dressian $\Dr(3,n)$.  The associated tropical plane $L_p$ in
$\TP^{n-1}$ is the intersection of the tropical hyperplanes
\[
\tropical{\,
p_{ijk}x_l + p_{ijl} x_k + p_{ikl} x_j + p_{jkl}x_i\,}
\]
as $\{i,j,k,l\}$ ranges over all $4$-element subsets of $[n]$.  By a
\emph{tropical plane} we mean any subset of $\TP^{n-1}$ which has the
form $L_p$ for some $p \in \Dr(3,n)$. The tropical plane $L_p$ is
realizable as the tropicalization of a classical plane in
$\PP_\K^{n-1}$ if and only if $p \in \Gr(3,n)$.  The plane $L_p$ is
called \emph{series-parallel} if each cell in the corresponding matroid
subdivision of $\Delta(3,n)$ is the graphic matroid of a
series-parallel graph.  Results of Speyer \cite{Spe1, Spe2} imply:

\begin{prop}
  Let $L$ be a tropical plane in $\TP^{n-1}$ with $f_0(L)$ vertices,
  $f_1^{b}(L)$ bounded edges, $f_1^{u}(L)$ unbounded edges,
  $f_2^{b}(L)$ bounded $2$-cells and $f_2^{u}(L)$ unbounded $2$-cells.
  Then
  \[
  \begin{matrix}
    f_0(L)       & \!\! \leq \!\! & (n-2)(n-3)/2, &
    f_1^{b} (L)  & \!\! \leq \!\! & (n-4)(n-3),\, &
    f_1^{u} (L)  & \!\! \leq \!\! & n(n-3), \\ &&&
    f_2^{b} (L)  & \!\! \leq \!\! & (n-4)(n-5)/2, &
    f_2^{u} (L)  & \!\! \leq \!\! & 3n(n-1)/2 \,.
  \end{matrix}
  \]
  These five inequalities are equalities if 
and only if $L$ is a series-parallel
  plane.
\end{prop}

The unbounded edges and $2$-cells of a tropical plane correspond to
the nodes and edges of the $n$ trees in the corresponding tree
arrangement. Suppose the trees are trivalent. Then each tree has $n-1$
leaves and $n-3$ nodes, for a total of $f_1^{u} (L) = n(n-3)$ nodes.
Moreover, each tree has $n-4$ interior edges and $n-1$ pendent edges.
The latter are double-counted.  This explains the number $f_2^u(L) =
n(n-4) + n(n-1)/2$ of edges in the tree arrangement representing $L$.
To understand this situation geometrically, we identify $\TP^{n-1}$
with an $(n-1)$-simplex, and we note that the tree arrangement is
obtained geometrically as the intersection
$\,L\cap\partial\TP^{n-1}\,$ of $L$ with the boundary of that simplex.

The \textbf{first answer} to our question of how to draw a tropical
plane is given by Theorem~\ref{thm:arrangement}: simply \textbf{draw
  the corresponding tree arrangement}.  This answer has the following
interpretation as an algorithm for enumerating all tropical planes.
To draw all (generic) planes $L$ in $\TP^{n-1}$, we first list all
trees on $n-1$ labeled leaves.  Each labeled tree occurs in $n$
relabelings corresponding to the sets
$[n]\backslash\{1\},[n]\backslash\{2\},\dots,[n]\backslash\{n\}$ of
labels.  Inductively, one enumerates all arrangements of $4,5,\dots,n$
trees.  This naive approach works well for $n\le 6$. The result of the
enumeration is that, up to relabeling and restricting to trivalent
trees, there are precisely seven abstract tree arrangements for $n=6$.
They are listed in Table~\ref{tab:trees}.  Each tree is written as
$ab\,c\, de$, the notation introduced in Figure~\ref{fig:tree5}.  We
then check that each of the seven abstract tree arrangements supports
a metric tree arrangement, and we conclude that $\Dr(3,6)$ has seven
maximal cells modulo the natural action of the group $\Sym_6$.  The
names for the seven types of generic planes are the same as in
\cite[\S 5]{SS} and in Figure \ref{fig:planes}.

\begin{table}[hbt]\centering
  \caption{The trees corresponding to the seven types of tropical planes in $\TP^5$.}
  \renewcommand{\arraystretch}{0.9}
  \begin{tabular*}{.9\linewidth}{@{\extracolsep{\fill}}lccccccr@{}}
    \toprule
    Type&Tree 1&Tree 2&Tree 3&Tree 4&Tree 5&Tree 6&Orbit Size\\
    \midrule
    EEEE&23 6 45&13 5 46&12 4 56&15 3 26&14 2 36&24 1 35&30\\
    EEEG&26 5 34&16 5 34&14 2 56&13 2 56&12 3 46&12 3 45&240\\
    EEFF(a)&25 6 34&15 6 34&12 5 46&12 5 36&12 6 34&12 5 34&90\\
    EEFF(b)&25 6 34&15 6 34&12 6 45&12 6 35&12 6 34&12 5 34&90\\
    EEFG&25 6 34&15 6 34& 24 1 56&23 1 56&12 6 34&12 5 34&360\\
    EFFG&34 2 56&34 1 56&12 6 45&12 6 35&12 6 34&12 5 34&180\\
    FFFGG&34 2 56&34 1 56&12 4 56&12 3 56&12 6 34&12 5 34&15\\
    \bottomrule
  \end{tabular*}
  \label{tab:trees}
\end{table}

It is easy to translate the seven rows in Table~\ref{tab:trees} into
seven pictures of tree arrangements. For example, the representative
for type ${\rm FFFGG}$ in the last row coincides with (\ref{eq:FFFGG})
and its picture appears on the left side in Figure \ref{fig:lozenge}.
It can be checked in the pictures that each of the seven tree
arrangements has $f_1^{u} (L) = 18$ nodes and $f_2^{u} (L) = 27$
edges.

The \textbf{second answer} to our question of how to draw a tropical
plane is given by Figure~\ref{fig:planes}: simply \textbf{draw and
  label the bounded cells}.  The planes $L$ in the last six rows of
Table~\ref{tab:trees} are series-parallel. Here, the complex of
bounded cells in $L$ has $f_0(L) = 6$ nodes, $f_1^b(L) = 6$ edges and
$f_2^b(L) = 1$ two-dimensional cell.  The first type ${\rm EEEE}$ is
not series-parallel: its bounded complex is one-dimensional,
with four edges and five nodes.

Each node of (the complex of bounded cells of) a tropical plane $L$ is
labeled by a connected rank-$3$-matroid. This is the matroid whose
matroid polytope is dual to that node in the matroid subdivision of
$\Delta(3,n)$ given by $L$.  For $n=6$ only three classes of matroids
occur as node labels of generic planes.  These matroids are denoted
$\{A,B,C,D\}$, $[A,B,C,D](E)$, or $\langle A;a;(b,c,d,e)\rangle$.
Here capital letters are non-empty subsets of and lower-case letters
are elements of the set $\{1,2,3,4,5,6\}$. All three matroids are
graphical.  The corresponding graphs are shown in
Figure~\ref{fig:graphs}.  Note that an edge labeled with a set of $l$
points should be considered as $l$ parallel edges each labeled with
one element of the set.

\begin{figure}[htb]
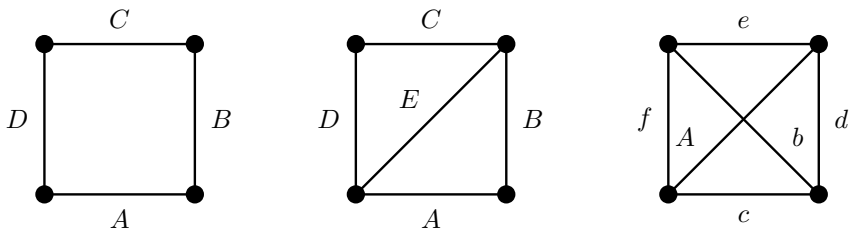
\centering
  \begin{minipage}[c]{.25\textwidth}\centering
    \includegraphics[scale=1]{graphs.0}
  \end{minipage}
  \begin{minipage}[c]{.25\textwidth}\centering
    \includegraphics[scale=1]{graphs.1}
  \end{minipage}
  \begin{minipage}[c]{.25\textwidth}\centering
    \includegraphics[scale=1]{graphs.2}
  \end{minipage}
  \caption{The graphic matroids corresponding to the labels
    $\{A,B,C,D\}$, $[A,B;C,D](E)$ and $\langle A;b;(c,d,e,f)\rangle$
    used for the nodes in Figure \ref{fig:planes}.}
  \label{fig:graphs}
\end{figure}

The underlying graph of the matroid $\langle A;b;(c,d,e,f)\rangle$ is
the complete graph $K_4$.  The set $A$ is a singleton, and thus its
automorphism group is the full symmetric group $\Sym_4$ of order 24
acting on the four nodes of $K_4$.  This matroid occurs in the unique
orbit of planes (of type EEEE) in $\TP^6$ whose bounded parts are not
two-dimensional. The series-parallel planes use only the matroids
$\{A,B,C,D\}$ and   $[A,B;C,D](E)$ for their labels.

The \textbf{third answer} to our question is the synthesis of the
previous two: \textbf{draw both} the bounded complex and the tree
arrangement.  The two pictures can be connected, by linking each node
of $L$ to the adjacent unbounded rays and $2$-cells. This leads to an
accurate diagram of the tropical plane $L$. The reader might enjoy
drawing these connections between the seven rows of
Table~\ref{tab:trees} and the seven pictures in
Figure~\ref{fig:planes}.

The analogous complete description for $n=7$ is a main contribution of
this paper. Based on the computational results in Section 2, we
prepared an online census of $\Gr(3,7)$ and $\Dr(3,7)$, with a picture
for each bounded complex. This is posted at our website
\begin{quote}\small
  \url{www.uni-math.gwdg.de/jensen/Research/G3_7/grassmann3_7.html}.
\end{quote}

\begin{figure}[htb]
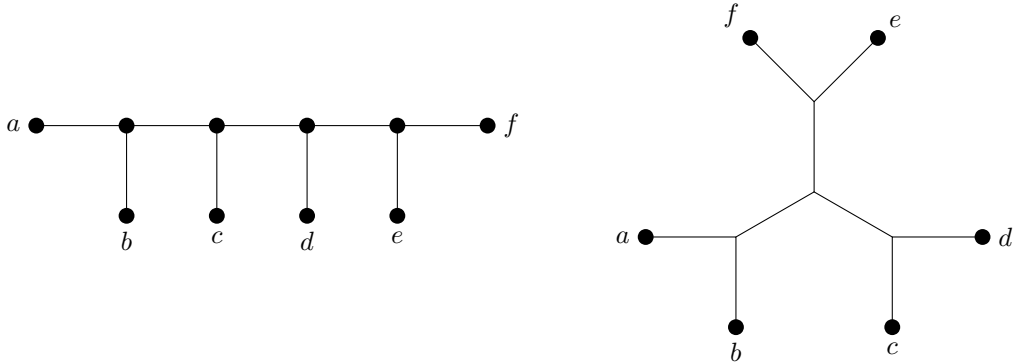
\centering
  \begin{minipage}[c]{.45\textwidth}\centering
    \includegraphics[scale=1]{tree.2}
  \end{minipage}
  \begin{minipage}[c]{.45\textwidth}\centering
    \includegraphics[scale=1]{tree.3}
  \end{minipage}
  \caption{Caterpillar tree C$(ab,cd,ef)$ and snowflake tree S$(ab,cd,ef)$.}
  \label{fig:trees6}
\end{figure}

The maximal cells of the Dressian $\Dr(3,7)$ correspond to
arrangements of seven trivalent trees. As part of our computations, we
found that for $n=7$ there is no difference between abstract tree
arrangements and metric tree arrangements: nothing like Example
\ref{exa:paco} exists in this case.  To draw the tree arrangements, we
note that there are two distinct trivalent trees on six leaves.  These
are the \emph{caterpillar} and the \emph{snowflake} trees depicted in
Figure~\ref{fig:trees6}.  Caterpillar trees exist for all $n\ge 5$,
and are encoded using a natural generalization of the notation in
Figure~\ref{fig:tree5}.  Note, for instance, the caterpillars with
eight leaves in Table~\ref{tab:paco}.

We conclude with a brief discussion of the $94$ generic planes
depicted on our website.  Four types of node labels occur in the
Dressian $\Dr(3,7)$.  First of all, the matroids $\{A,B,C,D\}$,
$[A,B,C,D](E)$, and $\langle A;b;(c,d,e,f)\rangle$ appear again.  Here
capital letters are non-empty subsets of and lower-case letters are
elements of $\{1,2,\ldots,7\}$.  The other matroid which occurs is the
Fano matroid $\cF_3$ arising from the projective plane $\PG_2(2)$; see
Figure~\ref{fig:non-fano} (left).  It corresponds to the
six-dimensional cells of $\Dr(3,7)$ generated by seven vertex splits.
Each such $6$-cell admits seven coarsenings arising from omitting one
of the seven splits.  These coarsenings correspond to the non-Fano
matroid; see Figure~\ref{fig:non-fano} (right).

\section{Restricting to Pappus}
\label{sec:pappus}

The Grassmannian $\Gr(d,n)$ is a tropical variety and the Dressian
$\Dr(d,n)$ is a tropical prevariety.  We now consider these two fans
inside the tropical projective space $\TP^{\binom{n}{d}-1}$.  That
projective space is a simplex, and it makes sense to study their
intersections with each (relatively open) face of
$\TP^{\binom{n}{d}-1}$. That intersection is non-empty only if the
face corresponds to a matroid $\cM$ of rank $d$ on $[n]$. This leads
to the following relative versions of our earlier definitions.

We define the
\emph{Grassmannian} $\Gr(\cM)$ \emph{of a matroid} $\cM$ 
to be the tropical variety
defined by the ideal $I_\cM$ which is obtained from the Pl\"ucker
ideal by setting to zero all variables $p_X$ where $X$ is not a basis of $\cM$.  We define the \emph{Dressian} $\Dr(\cM)$ to be
the tropical prevariety given
by the set of quadrics which are obtained from the quadratic Pl\"ucker
relations by setting to zero all variables $p_X$ where $X$ is not a basis of $\cM$.  Equivalently, in the language of \cite{DW1,DW2}, the
Dressian $\Dr(\cM)$ is the set of all real-valued valuations of the
matroid $\cM$.  As before, $\Gr(\cM)$ is a subfan of the Gr\"obner fan
of $I_\cM$, the Dressian $\Dr(\cM)$ is a subfan of the secondary fan
of the matroid polytope of $\cM$, and we regard these fans as polyhedral
complexes after removing the lineality space and intersecting with a sphere.
Note that the cells of
$\Dr(\cM)$ are in bijection with the regular matroid subdivisions of
the matroid polytope of $\cM$.  The Grassmannian $\Gr(d,n)$ and the
Dressian $\Dr(d,n)$ discussed in the previous sections are special
cases where $\cM$ is the uniform matroid of rank $d$ on $n$ elements.
The Dressian $\Dr(d,n)$ contains the Dressians of all matroids of rank
$d$ on $n$ elements as subcomplexes at infinity.

In this final section we examine these concepts in detail for one
important example, namely, we take $\cM$ to be the \emph{Pappus
  matroid}. Here $d=3$, $n=9$, $\cM$ has 75 bases, and the non-bases
are the nine lines in the Pappus configuration shown in
Figure~\ref{fig:pappus}:
\[
123, \ 148, \ 159, \ 247, \ 269, \ 357, \ 368, \ 456, \ 789 \, .
\]
The ideal $I_\cM$ is the ideal in the polynomial ring in 75 variables
obtained from the Pl\"ucker ideal by setting the corresponding
nine Pl\"ucker coordinates to zero: $p_{123} = \cdots = p_{789} = 0$.

\begin{figure}[htb]
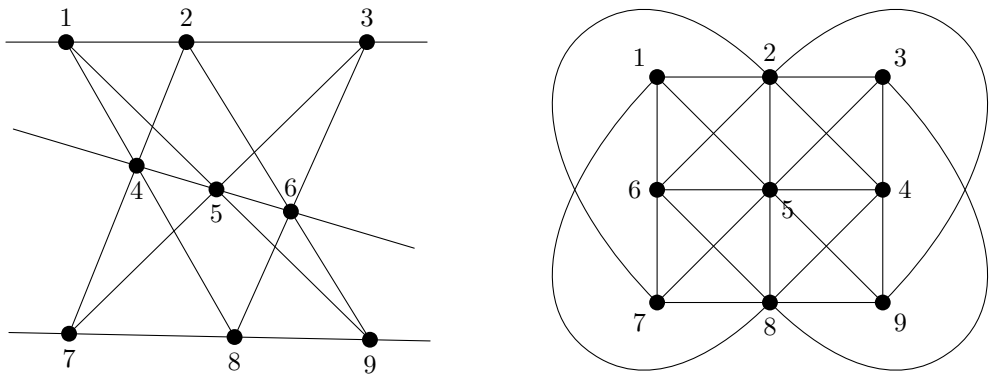

  \begin{minipage}[c]{.45\textwidth}\centering
    \includegraphics[scale=1]{pappus.0}
  \end{minipage}
  \begin{minipage}[c]{.45\textwidth}\centering
    \includegraphics[scale=1]{hessian.0}
  \end{minipage}
  \caption{Pappus configuration (left) and Hessian configuration (right).}
  \label{fig:pappus}
\end{figure}

The realization space of the Pappus configuration modulo projective
transformations is two-dimensional, and the Grassmannian $\Gr(\cM)$ is
the corresponding tropical surface. We shall determine the underlying
graph and how it embeds into the Dressian $\Dr(\cM)$.

\begin{prop} The Grassmannian  $\Gr(\cM)$ of the Pappus matroid 
$\cM$ is a graph with $19$ nodes and $30$ edges. One of
the nodes gets replaced by a triangle in the
Dressian  $\Dr(\cM)$.
The Dressian $\Dr(\cM)$ is a simplicial complex with $18$ vertices, $30$
edges and one triangle. 
\end{prop}

\begin{proof}
  What follows is a detailed description, first of $\Gr(\cM)$ and
  later of $\Dr(\cM)$.  The Grassmannian $\Gr(\cM)$ has three
  \emph{split nodes}, represented by the bases $167$, $258$ and $349$
  of the Pappus matroid $\cM$.  These three bases are characterized by
  the property that their two-element subsets form two-point lines.
  The corresponding matroid subdivisions are vertex splits, and they
  are the only splits of the matroid polytope $\MatroidPolytope{\cM}$.
  The three split vertices have valence four, and they are connected
  to a special trivalent \emph{core node} $C$.

  \begin{figure}[ht]\centering
    \includegraphics[width=.5\textwidth]{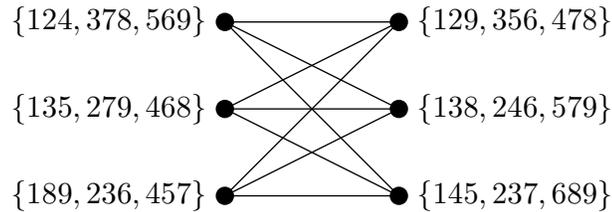}
    \caption{Complete bipartite graph formed from the Graves triads.}
    \label{fig:graves}
  \end{figure}

  The remaining 15 nodes are all trivalent in $\Gr(\cM)$, and their subgraph
  corresponds to the vertices and edges of the complete bipartite graph
  $K_{3,3}$. The six vertices of $K_{3,3}$ correspond to six \emph{Graves
    nodes}, one for each of the Graves triads in the Pappus configuration. A
  \emph{Graves triad} is a partition of the nine points into three bases whose
  two-element subsets span three-point lines.  Each Graves node defines a
  matroid subdivision with three maximal cells. The three corresponding
  matroids have $52$ bases, and they are obtained geometrically by merging
  together the three points in a triple of the Graves triad. For example, the
  first matroid in the subdivision defined by the Graves triad $\{145, 237,
  689\}$ is obtained from the Pappus matroid by making 1, 4 and 5 parallel
  elements.

  The six Graves triads form the vertices of the graph $K_{3,3}$ shown
  in Figure~\ref{fig:graves}.  On each of the nine edges lies a
  \emph{connector node} of $\Gr(\cM)$, which is between two Graves
  nodes and also adjacent to one of the three split nodes.  Each
  connector node defines a matroid subdivision with seven maximal
  cells.  The number of bases of these seven matroids are $36$, $36$,
  $36$, $40$, $40$, $40$, $51$.  For a concrete example consider the
  two adjacent Graves triads $\{145, 237, 689\}$ and $\{189, 236,
  457\}$. On the edge between them in $K_{3,3}$ we find a connector
  node which is also adjacent to the split node $167$. The seven
  matroids in the matroid subdivision of $\MatroidPolytope{\cM}$
  represented by that connector node are the rows in the following
  table:

  \begin{table}[ht]\centering
    \begin{tabular*}{.5\linewidth}{@{\extracolsep{\fill}}cc@{}}
      \toprule
      number of bases  &   parallelism classes \\
      \midrule
      51          &   \{8, 9\}, \{2, 3\}, \{4, 5\}\\
      40          &        \{2, 3, 6, 7\}\\
      40          &        \{1, 4, 5, 7\}\\
      40          &        \{1, 6, 8, 9\}\\
      36          &   \{4, 5, 7\}, \{6, 8, 9\}\\
      36          &   \{1, 4, 5\}, \{2, 3, 6\}\\
      36          &   \{1, 8, 9\}, \{2, 3, 7\}\\
      \bottomrule
    \end{tabular*}
  \end{table}
  
  We now come to the Dressian $\Dr(\cM)$ of the Pappus matroid $\cM$.
  This is a non-pure complex whose facets are one triangle and $27$
  edges.  It is obtained from $\Gr(\cM)$ by removing the core
  node and replacing it with the \emph{core triangle} whose nodes are
  the split nodes $167$, $258$ and $349$.  Thus $\Dr(\cM)$ has $18$
  vertices, $30$ edges and one triangle. The core triangle of $\Dr(\cM)$
  represents the matroid subdivision which is obtained from the Pappus
  matroid polytope by slicing off the three vertices $167$, $258$ and
  $349$.  What remains is the matroid polytope of the Hessian
  configuration shown in Figure~\ref{fig:pappus}. This is the matroid
  associated with the affine plane over the field $\GF(3)$ with three
  elements.  Collinearity of any eleven of its twelve triples implies
  collinearity of the last. It is this incidence theorem which
  explains the difference between $\Gr(\cM)$ and $\Dr(\cM)$. An
  algebraic witness is offered by the expression
  \[
  p_{289} p_{389} p_{489} p_{569} p_{589} \underline{p_{167}}
  - p_{189} p_{389} p_{489} p_{569} p_{679} \underline{p_{258}}
  + p_{189} p_{289} p_{569} p_{589} p_{678} \underline{p_{349}}\, .
  \]
  This trinomial lies in the Pappus ideal $I_\cM$, and it shows that
  the tropical variety of $I_\cM$ does not contain the entire
  triangular cone spanned by the basis vectors $e_{167}, e_{258},
  e_{349}$. As the minimum must be attained at least twice, we
  conclude that, locally on the core triangle of the Dressian
  $\Dr(\cM)$, the Grassmannian $\Gr(\cM)$ looks like a tropical line.
\end{proof}

\bigskip

\end{document}